\documentclass[11pt]{amsart}

\usepackage[utf8]{inputenc}
\usepackage[T1]{fontenc}
\usepackage{amsmath, amssymb, amsthm}
\usepackage{geometry}
\usepackage{hyperref}

\usepackage[utf8]{inputenc}
\usepackage{amsmath,amsfonts,amsthm,amssymb}
\usepackage{enumitem}
\usepackage{comment}
\usepackage{graphicx}
\usepackage{mathrsfs}
\usepackage{fancyhdr}
\usepackage{lastpage}
\usepackage [english]{babel}
\usepackage[autostyle]{csquotes}
\usepackage{tikz-cd}
\usepackage{bbm}
\usepackage{booktabs}

\usepackage{comment}

\usepackage[
backend=biber,
style=numeric,
sorting=ynt
]{biblatex}
\newtheoremstyle{colon}{}{}{\itshape}{}{\bfseries}{}{ }{}
\newtheoremstyle{pcolon}{}{}{}{}{\itshape}{}{ }{}

\theoremstyle{definition}
\newtheorem{definition}{Definition}[section]

\newtheorem*{definition*}{Definition}

\newtheorem{theorem}[]{Theorem}[section]
\newtheorem*{theorem*}{Theorem}
\newtheorem{proposition}{Proposition}[section]

\newtheorem*{proposition*}{Proposition}

\newtheorem{lemma}{\itshape{Lemma}}
\theoremstyle{colon}

\newtheorem*{lemma*}{\itshape{Lemma: }}

\theoremstyle{pcolon}
\newtheorem*{claim*}{Claim:}
\theoremstyle{remark}
\newtheorem*{proofclaim*}{Proof of Claim}

\renewcommand\phi{\varphi}

\newcommand{\RR}{\mathbb{R}}

\newcommand{\SSp}{\mathbb{S}}

\newcommand{\cinf}{\mathcal{C}^{\infty}}

\newcommand{\set}[1]{\{ #1\}}

\newcommand{\lieG}{\mathfrak{g}}

\newcommand{\SES}[3]{0\longrightarrow #1 \longrightarrow #2 \longrightarrow #3 \longrightarrow 0}

\title[\MakeUppercase{Stratification of the Jeffrey-Weitsman-Witten invariants}]{\MakeUppercase{Stratification of the half-density quantization of the Jeffrey-Weitsman-Witten invariants}}

\author{Adrian Chitan}
\address{Department of Mathematics, Western University, }
\email{achitan@uwo.ca}

\date{} 

\begin{document}

\begin{abstract}
    In this paper we review the quantization procedure employed by Jeffrey and Weitsman in their formalization of Witten's three manifold invariants. In so doing, we formalize the stratification of the Lagrangian leaf which corresponds to the flat connections of a bounding surface which extend into the handlebody. The smooth structure of this space motivates a procedure for a stratified quantization: use volumes and corresponding half-densities on the stratified tangent bundle. The pairing is then defined such that the proposed BKS pairing still yields the torsion in the integrand of the Jeffrey-Weitsman-Witten invariant. The stratification yields a natural covariant constant half-density which results in a reduction of the structure group for reducible connections. The conditions of their integral are refined in the context of the stratification and this change reproduces the invariant exactly for a plethora of 3-manifolds and extends the invariant for some.
\end{abstract}
\maketitle
\section{Introduction}
In \cite{WITTENQFT} Witten defined the invariant via the Feynman path integral: $$Z(N,k) = \int d\mathscr{A} e^{2\pi i k CS (A)}$$

\noindent where $CS(A)$ is the Chern-Simons invariant of the connection on the trivial $G=SU(2)$ principal bundle over a closed oriented 3-manifold $N$. This integral is over Gauge equivalent orbits in the space of connections, however the measure $d\mathscr{A}$ is not defined. The method of stationary phase has been explored in Witten's paper and also by Freed and Gompf in \cite{Freed:1991wd}. The crucial point which begins this work into the moduli space of flat connections then comes from the fact that the singular points of the Chern-Simons functional are precisely the flat connections.
\subsection{The Freed \& Gompf calculations and conjecture}
Let us restate the perturbation theory result as was refined by Freed and Gompf. In the case of isolated classes of Gauge equivalent flat connections it is: 
\begin{equation}{\tag{FG 1.32}}
    Z(N,k) \sim \frac{1}{2}e^{3\pi i/4 }\sum_i \sqrt{\tau (A_i)} e^{-2\pi i I_{A_i}/4}e^{2\pi i CS(A_i)(k+2)} 
\end{equation}

 where the sum is over these isolated classes. $I_{A_i}$ is associated to the spectral flow of the connection and $\tau(A_i)$ is the absolute value of the torsion of the connection (\emph{the Reidemesiter torsion}, see \cite{Freed1992} for a full background.) 
\subsection{The half-density quantization of Jeffrey-Weitsman}

 In 1992, Jeffrey and Weitsman employed geometric quantization with half-densities to the moduli space of flat connections, a space equipped with -- at most points -- the Atiyah-Bott symplectic form, the Chern-Simons prequantum line bundle and a real \enquote{polarization}. 

The setup comes from topological quantum field theory: a closed oriented 3-manifold admits a Heegard decomposition -- that is, $$N=H_g^1\sqcup_{\Sigma^g}H_g^2,$$
a splitting along some genus $g$ surface, $\Sigma^g$, into two genus $g$ handlebodies, $H^1_g$ and $H^2_g$, with a diffeomorphism between their boundaries under which they are glued to form $N$. 

The polarization we shall see has a degenerate fibre, $L_{H^i_g}$, which corresponds to flat connections on the surface which extend to flat connections on the bounding handlebody modulo gauge transforms. This space is not smooth. We will explore its stratified nature which will be determined by the connection's stabilizer: either all of $SU(2)$ for \enquote{trivial} connections, the center of $SU(2)$ for irreducible connections, or some maximal torus for non-trivial reducible connections. The interpretation of the moduli space as the character variety is crucial: it allows the view of $L_{H^i_g}$ as $SU(2)^g$ modulo the proper action of $SU(2)$ by simultaneous conjugation. The natural bi-invariant Riemannian structure induced from $SU(2)$ onto the strata of the $L_{H^i_g}$ gives the primary justification for the covariant constant nature of the volumes and half-densities.

\section{The Moduli Space of flat connections}
\subsection{The infinite dimensional construction}
Let $\Sigma^g$ denote the genus $g$ surface, $G=SU(2)$ and consider the trivial principal bundle $\Sigma\times G$. Recall that the principal connections can be associated with the Lie algebra valued $1$-forms on $\Sigma$, denoted $\Omega^1(\Sigma,\mathfrak{g})$, since the principal bundle is trivial. There is a natural pairing given by the Atiyah-Bott form \cite{atiyah1983yang} and further discussed in \cite{GOLDMAN1984200}:
$$(\alpha,\beta)\overset{\tilde{\omega}}{\longmapsto}\int_\Sigma Tr(\alpha \wedge \beta)$$
where $\alpha,\beta \in \Omega^{1}(\Sigma,\mathfrak{g})$. The Gauge group, $Gau(G)$, can be identified with $\cinf(\Sigma, G)$ and acts via:
$$h\cdot A = h^{-1}Ah + h^{-1}dh$$
where the multiplication on the right hand side is matrix multiplication of these matrix valued terms. This action can be shown to be Hamiltonian, with moment map:
$$\mu(A) = F_A$$
being the curvature of the connection. In particular, the symplectic reduction at level zero, modulo the Gauge group action, is denoted:
$$\mathscr{S}_g := \mu^{-1}(0)_{//Gau(G)}.$$
This incarnation of the moduli space, as the symplectic quotient of an infinite dimensional affine space carries many advantages, such as quickly seeing the Atiyah-Bott form is closed (since the form $\tilde \omega_A$ has no dependence on the connection $A\in \Omega^1(\Sigma, \mathfrak{g})$), thus making $\mathscr{S}_g$ a symplectic manifold.  
\subsection{Representations of the fundamental group}
In this section, we look into an alternative space which will be isomorphic to the moduli space of flat connections (modulo the Gauge group action). That is, the space of representations of the fundamental group into the structure group, $G$, modulo conjugation. Explicitly, starting with a nice presentation of the fundamental group of the surface bounding a genus-g handlebody, $H_g$:
\begin{align*}
    \pi_1(\Sigma) &= <A_1,\dots,A_g,B_1,\dots,B_g|A_1B_1A_1^{-1}B_1^{-1}\dots A_gB_gA_g^{-1}B_g^{-1}>\\
    \intertext{where the $A_i$'s generate $\pi_1(H_g)$, and then define:}
    R_g:&= \set{\tilde{x} = (x_1,\dots ,x_g, y_1,\dots , y_g)\in G^{2g} : \Pi_i [x_i,y_i] = 1}
    \intertext{where $[\cdot, \cdot]$ denotes the matrix commutator. Now by considering the image of these generators under any chosen representation, $\rho:\pi_1(\Sigma)\to G$, we have the identification:}
    R_g &\cong Hom(\pi_1(\Sigma), G)
\end{align*}
under the mapping $\rho \mapsto (\rho(A_1),\dots , \rho(B_g))$. There is an action of $G$ on $R_g$ by simultaneous conjugation, with orbit space denoted by $\bar{\mathscr{S}}_g$. This suggestive notation is justified by the following theorem, which establishes an isomorphism of $\mathscr{S}_g$ with the set of irreducible representations in $R_g$.
\begin{theorem*}
    The moduli space of flat connections is symplectomorphic to an open dense subset of $\bar{\mathscr{S}}_g$ consisting precisely of the irreducible representations. 
\end{theorem*}
A formal proof is deferred to \cite{morita2001geometry}, however, starting with a flat connection one can easily define a representation of the fundamental group by considering the holonomy of the connection on any of the loops generating the chosen presentation above -- this assignment does descend to a representation of $\pi_1(\Sigma)$, as the holonomy is invariant under free homotopy of loops when a connection is flat (this is seen through the Stoke's theorem).  On the other hand, if one starts with a representation, we can ascend to the universal cover of $\tilde{\Sigma}$ over the surface, and consider the trivial connection of the bundle $\tilde{\Sigma}\times G$. On this space, there is an action of the fundamental group twisted by a fixed representation, $\rho$:
$$\alpha \cdot(p, h) = (\alpha\circ p, \rho(\alpha)^{-1}h)$$
where $\alpha\in \pi_1(\Sigma), p\in \tilde{\Sigma}, h\in G$ with the action on the universal cover being the deck transform determined by $\alpha$. The trivial (flat) connection, then descends to a flat connection on $\Sigma \times G$. These correspondences can be shown to be inverses, if connections are considered modulo the Gauge group action and representations which are irreducible are considered modulo simultaneous conjugation by $G$.

The symplectic form can also be defined in the representation theoretic view, for this we will use group cohomology, keeping in mind that $\Sigma$ is a $K(\pi_1(\Sigma), 1)$ space and so group cohomology and singular cohomology can be identified with the relevant coefficients. More generally, we use $\pi_1$ to denote the fundamental group of a manifold, and when necessary, we specify to $\pi_1(\Sigma^g)$. Explicitly, we derive the tangent space to the smooth part of $Hom(\pi_1, G)/G$. Fix any representation $\tilde{x}\in Hom(\pi_1, G)$ and denote by $\mathfrak{g}_{\tilde{x}}$ the Lie algebra of $G$ equipped with the $\pi_1$ module structure, via the multiplication:
$$\alpha \cdot X = Ad_{\tilde{x}(\alpha)}X$$
for $\alpha \in \pi_1$ and $X\in \lieG$. Following Goldman (\cite{GOLDMAN1984200}) we can consider local one parameter curves: 
\begin{align*}
    \phi_t(\alpha) = e^{tu(\alpha)+O(t^2)}\phi(\alpha)
\end{align*}
where $\alpha\in \pi_1$ and subject to the homomorphism condition $\phi_t(\alpha\beta ) = \phi_t(\alpha)\phi_t(\beta)$ forces $u$ to be a 1-cocycle, and conversely a 1-cocycle can be realized as a derivation using a homomorphism defined in this way. We will soon explore the non-smoothness of the moduli space, but for now we will say $Z^1(\pi_1(\Sigma),\mathfrak{g}_{\tilde{x}})$ is the Zariski tangent space to $Hom(\pi_1(\Sigma), G)$. Recalling the mapping $d_{\tilde{x}}$ of $0$-cochains into $1$-cochains, defined in group cohomology is explicitly given by $$d_{\tilde{x}}(\xi)(\alpha) = Ad_{\tilde x(\alpha)}\xi - \xi$$ 
where $\xi \in \mathfrak{g}, \alpha \in \pi_1(\Sigma)$, shows for $\tilde{x}$ irreducible, $B^1(\pi_1(\Sigma), \mathfrak{g}_{\tilde x}) = \mathfrak{g}_{\tilde x}$. With this fact, the simultaneous conjugation of $Hom(\pi_1(\Sigma), G)$ corresponds locally to quotienting out the infinitesimal group action -- the adjoint action. We summarize: 
\newcommand{\pixt}[0]{\pi_1(\Sigma), \mathfrak{g}_{\tilde{x}}}
\begin{proposition*}
    The Zariski tangent space to $Hom(\pi_1(\Sigma), G)$ is given by $Z^1(\pixt)$ as part of the sequence:
    $$\SES{B^1(\pixt)}{Z^1(\pixt)}{H^1(\pixt)}$$
    where tangent space to $Hom(\pi_1(\Sigma), G)/G$ at $[\tilde{x}]$ for $\tilde{x}$ irreducible is isomorphic to $$H^1(\pi_1(\Sigma), \mathfrak{g}_{\tilde{x}})\cong H^1(\Sigma, \mathfrak{g}_{\tilde{x}}).$$
\end{proposition*}

\section{Some invariants of connections}

\section{The polarization and space of constant dimension}
The surface of genus $g$ has trinion (a thrice punctured sphere) decompositions into $2g-2$ trinions. An example of one such decomposition in genus $4$ is shown below:\\
\includegraphics[width=\textwidth]{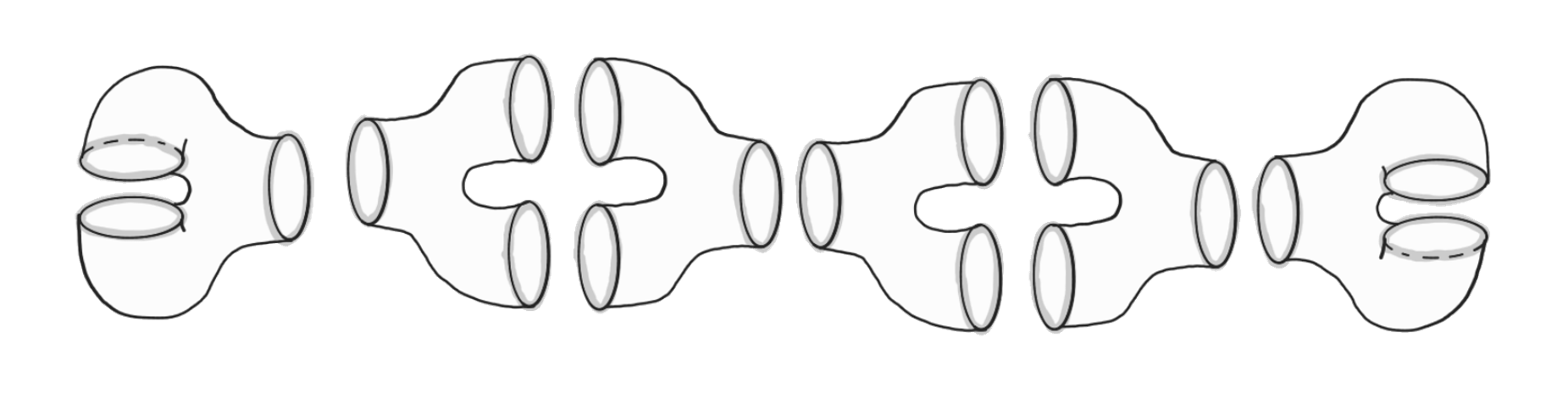}
Such a decomposition gives rise to $3g-3$ non-trivial loops highlighted below:\\
\includegraphics[width=\textwidth]{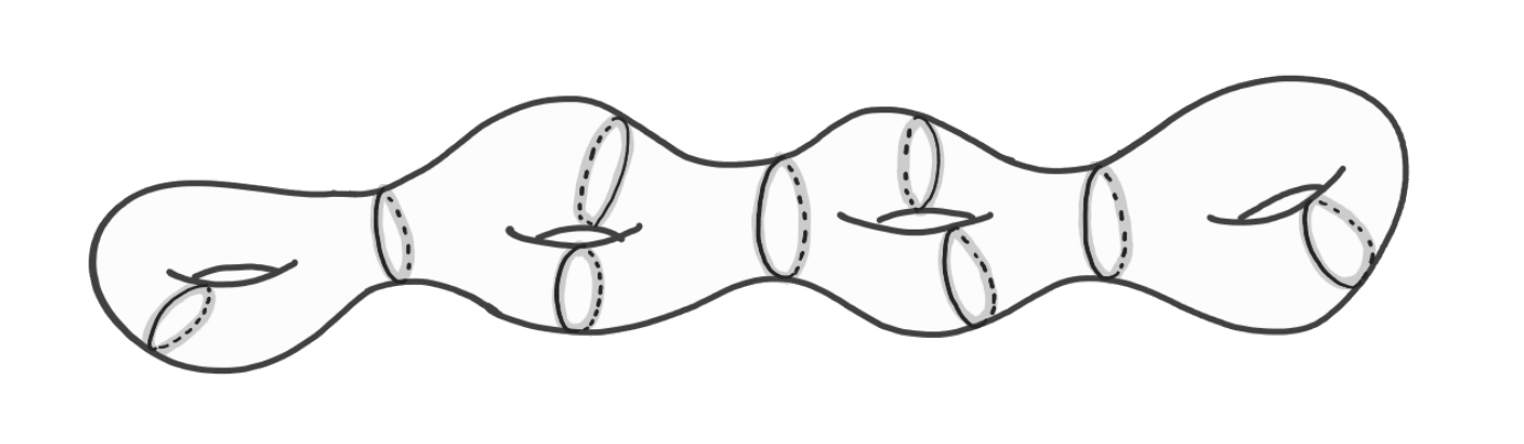}
Since the surface is path connected, we fix a base point to which we can connect to each of these loops such that they correspond to fundamental group classes of mutual base point. We denote these classes of loops by $\set{C_i}_{i=1}^{3g-3}$. This gives rise to functions $f_i:R_g\to \RR$ defined by:
$$\tilde{f}_i(\rho) = tr(hol_\rho([C_i]))$$
where $\rho\in R_g$, $hol_\rho$ denotes the holonomy of the representation about a loop -- note that this is well-defined on fundamental group classes since the representation's associated connection is flat. These functions then descend to $f_i:\bar{\mathscr{S}}_g\to \RR$ through the quotient by the simultaneous conjugation, since the trace is invariant. We thus have a continuous mapping:
$$\pi: \bar{\mathscr{S}}_g\to \RR^{3g-3}$$
defined as $\pi = (f_1,\dots, f_{3g-3})$. The content of the following theorem asserts this map is an honest polarization at regular points of the fibration, $\pi$, but also at non-fibration points, tangent vectors to any smooth parts of these fibres are also annihilated by the symplectic form. 
\begin{theorem}
    (\cite{WEITSMAN1991} and 2.1 of \cite{JEFFREY1993509}) For any $b\in \RR^{3g-3}$, $\pi^{-1}(b)$ is an isotropic subvariety of $\bar{\mathscr{S}}_g$, and for $b\in Interior(im(\pi))$, $\pi^{-1}(b)$ is a Lagrangian subvariety. 
\end{theorem}
The first of these references resorts to explicitly defining Hamiltonian vector fields which are cleverly extended to the infinite space of connections, whereas the second reference introduces the adapted to trinion (a.t.d) connection to give a more Lie theoretic proof. The details are deferred to these sources.  
\subsection{Flat connections extending to a Handlebody}
Of central importance to the current endeavor, is the maximally degenerate fibre:
$$L_H:=\pi^{-1}(2,\dots, 2)$$
Notice that the points in $L_H$ are precisely those connections with trivial holonomy on the $C_i$ -- these will be viewed as flat connections of the surface which extend to flat connections on the bounding handlebody, $H_g$, whose fundamental group has presentation associated to the surface's:
$$\pi_1(H_g) = <A_1,\dots, A_g>$$
with the representation space, modulo simultaneous conjugation:
$$L_H=Hom(\pi_1(H_g), G)/G\cong G^g/G.$$
This fibre is not smooth, it contains irreducible representations and reducible ones. A representation $\rho \in Hom(\pi_1(H_g), G)$ is irreducible if and only if its stabilizer under the simultaneous conjugation action is the center, $Z(G)$. Fix a maximal torus $T\subseteq G$. There are strata based on the stabilizers which are the following subsets of $L_H \subseteq \bar {\mathscr{S}_g}$:
$$L_H = L_H^0 \sqcup L_H^1 \sqcup L_H^3$$
defined as: $$L_H^i = \set{x\in \bar {\mathscr{S}_g} | dim(Stab(\Tilde{x}))  = 3-i},$$ that is, the subsets with stabilizers being $G, T_x, Z(G),$ and so $L_H^3$ is the familiar smooth dense subset $\mathscr{S}_g \subseteq \bar {\mathscr{S}_g}$. 


\begin{lemma}
    Each of the $L_H^i$ are smooth with respective dimensions $0,g,3g-3$ for $i=0,1,3$, respectively.
\end{lemma}
\begin{proof}
    The arguments of \cite{atiyah1983yang} follow immediately to show smoothness of $L_H^3$. The isolated trivial connection is smooth. We thus focus on $L_H^1$. Let $N(T)$ be the normalizer of the maximal torus $T$, and the Weyl group be $W(T) = W$. We have the diagram:
\[\begin{tikzcd}
	{{T^g}^* \times G} \\
	\\
	{{T^g}^* \times G/T} && {\tilde{L}_H^1}
	\arrow[from=1-1, to=3-1]
	\arrow["{(\tau, h)\mapsto h\tau h^{-1}}", from=1-1, to=3-3]
	\arrow["{\tilde{\phi}}", dashed, from=3-1, to=3-3]
\end{tikzcd}\]
where we see the smooth descent of the simultaneous conjugation to the quotient map $\tilde{\phi}$. Now the Weyl group acts freely and properly on $G/T$, giving a mapping 
\begin{equation}\label{secondStratum}
    \phi: {{T^g}^* \times G/N(T)} \to {\tilde{L}_H^1}
\end{equation}
which is smooth but also a set bijection and hence a diffeomorphism. We now consider the quotienting of this realization of ${\tilde{L}_H^1}$ in two steps, first by $G/N(T)$, so ${T^g}^*\cong {(\tilde{L}_H^1})_{/G_{/N(T)}}$ and finally then by $N(T)$, to have ${T^g}^*/N(T) \cong {T^g}^*/W$. 
\end{proof}
In particular, we have locally the tangent space is  $T_xL_H^1\cong Lie(T)^g$.
\subsection{The Riemannian volume on the stratification}
 As expounded fully in \cite{PflaumStratified}, there is a natural (stratified) Riemannian structure on $L_H$ coming from the natural bi-invariant metric on $G$ -- this is Pflaum's theorem 4.2.4 where we note the action by simultaneous conjugation is proper and by isometries. The Riemannian volumes per stratum are given infinitesimally by the exact sequences:
 \begin{equation}
     \SES{\lieG}{\lieG^g}{T_xL^3_H}
 \end{equation}
 for $x\in L^3_H$. For $x\in L^1_H$, using the $G/T$ bundle structure of \ref{secondStratum} we have:
 \begin{equation}
     \SES{\lieG/\mathfrak{t}}{\mathfrak{t}^g\oplus \lieG/\mathfrak{t}}{TL_H^1\cong \mathfrak{t}^g}.
 \end{equation}
 For the stratum of isolated points, it is understood that a volume refers to a constant.
\subsection{The space of constant dimension}
    We now recall the space of constant dimension employed by Jeffrey-Weitsman which inspired their definition of the volumes -- it was motivated precisely because the defining sequences  also yields the Reidemesiter torsion of the connection on the handlebody.

    With the notation $F_g:=\pi_1(H_g)$, which is the free group on $g$ generators, the definition of the cohomology groups gives the short exact sequence:
$$0\longrightarrow B^1(F_g, \mathfrak{g}_{\tilde{x}}) \longrightarrow Z^1( F_g, \mathfrak{g}_{\tilde{x}})\longrightarrow H^1(F_g, \mathfrak{g}_{\tilde{x}})\longrightarrow 0$$
and we have a second short exact sequence coming from viewing the $0-$cocycles in $\mathfrak{g}$ which surjects onto the $1-$coboundaries via the co-differential:
\begin{equation}\label{eqn:constant}
    0\longrightarrow H^0 (F_g, \mathfrak{g}_{\tilde{x}}) \longrightarrow \mathfrak{g}\longrightarrow B^1(F_g, \mathfrak{g}_{\tilde{x}}) \longrightarrow 0.
\end{equation}
We can concatenate these to give:
$$0\longrightarrow H^0 (F_g, \mathfrak{g}_{\tilde{x}}) \longrightarrow \mathfrak{g}\longrightarrow Z^1( F_g, \mathfrak{g}_{\tilde{x}})\longrightarrow H^1(F_g, \mathfrak{g}_{\tilde{x}})\longrightarrow 0$$
\begin{proposition*}
    Restricting now to $\mathscr{L}_H$, the virtual vector space: $$X(F_g, \mathfrak{g}_{\tilde{x}}) := H^1 (F_g, \mathfrak{g}_{\tilde{x}})\ominus H^0 (F_g, \mathfrak{g}_{\tilde{x}})$$
    has constant dimension $3g-3$.
\end{proposition*}
\begin{proof}
    We can replace the $1-$cocycles with $\mathfrak{g}^g$ as the tangent space to $G^g$, which is of fixed dimension $gdim(G)$. We thus have the exact sequence:
    $$0\longrightarrow H^0 (F_g, \mathfrak{g}_{\tilde{x}}) \longrightarrow \mathfrak{g}\longrightarrow \mathfrak{g}^g\longrightarrow H^1(F_g, \mathfrak{g}_{\tilde{x}})\longrightarrow 0$$
    from which we get the short exact sequence:
    $$0\longrightarrow V\longrightarrow \mathfrak{g}^g\longrightarrow H^1(F_g, \mathfrak{g}_{\tilde{x}})\longrightarrow 0$$
    where $V= \mathfrak{g}/H^0 (F_g, \mathfrak{g}_{\tilde{x}})$. Which gives the dimension count:
    $$dim(G) - dim(H^0) + dim(H^1) = gdim(G).$$
    In other words, $dim(H^1)-dim(H^0)$ is of constant dimension, independent of $\tilde{x}$.
\end{proof}
We will make repeated use of the following lemma for manipulating our volumes:
\begin{lemma}
    For an exact sequence of vector spaces: 
    $$0\longrightarrow V_1\longrightarrow\dots\longrightarrow V_n\longrightarrow 0$$
    there is an isomorphism of tensors of determinant bundles:
    $$\bigotimes_{i :odd}\Lambda^{max}(V_i)\cong \bigotimes_{i:even}\Lambda^{max}(V_i).$$
    The result is the same for the corresponding density and half-density bundles.
\end{lemma}

From the exact sequence \eqref{eqn:constant} used above, the lemma says:
\begin{equation}\label{eqn:iso}
    \Lambda^{max}(H^0(F_g, \mathfrak{g}_{\tilde{x}}))\otimes \Lambda^{max} (\mathfrak{g}^g) \cong \Lambda^{max}\mathfrak{g}\otimes \Lambda^{max} (H^1(F_g, \mathfrak{g}_{\tilde{x}})).
\end{equation}
Recalling the group structure ($\otimes$) on line bundles we propose a line bundle in place of the undefined determinant line bundle over $L_H$ as:
$$\Lambda^{max}(L_H)_x:= \Lambda^{max} (H^1(F_g, \mathfrak{g}_{\tilde{x}})) \otimes \Lambda^{max}(H^0(F_g, \mathfrak{g}_{\tilde{x}}))^*$$
which gives the point wise definition, and via \eqref{eqn:iso} we have the trivial bundle structure through: 
\begin{equation*}\label{eqn:linebundle}
    \Lambda^{max}(L_H)\cong \Lambda^{max} (\mathfrak{g}^g) \otimes (\Lambda^{max}\mathfrak{g})^*
\end{equation*}
    Fixing a bivariant metric on $\mathfrak{g}$, we can produce a volume on $\mathfrak{g}$, denoted as $\eta(\mathfrak{g})$ being the Riemannian volume (over the manifold $G$), as well as $\eta(\mathfrak{g})^g\in \Lambda^{max}(\mathfrak{g}^g)'$ as the Riemannian volume on $G^g$. We have an adaptation of Jeffrey-Weitsman's proposition 4.4, with $v(\tilde{x})$ being the tensor of these two volumes:
\begin{proposition}
    For any $\tilde{x}\in G^g$, there exists a volume element: $$w(\tilde{x}) \in \Lambda^{max}H^1(F_g, \mathfrak{g}_{\tilde{x}})'\otimes \Lambda^{max}H^0 (F_g, \mathfrak{g}_{\tilde{x}}).$$ Associated to $w(\tilde{x})$ is a unique continuous half-density, $\rho \in \Gamma (\Lambda^{1/2}(L_H))$ for which this volume and half-density are smooth sections over each of the smooth strata of $L_H$ individually.
\end{proposition}
    From the preceding discussion, we have such a volume before the simultaneous conjugation, and so this volume descends smoothly to the smooth strata since $G$, when acting by simultaneous conjugation, acts infinitesimally by isometries on the tangent spaces (the adjoint action). The associated half-density is denoted by $\rho$.

This volume of Jeffrey-Weitsman is in analogy with covariant constant sections of generic fibres of the polarization, $\pi$. Letting $\mathfrak{l}_{\tilde{x}}:=Lie(Stab(\tilde{x}))$, we now rebuild this same volume, using a few extra isomorphisms which split the Lie algebras of the stabilizers from their orthogonal complement for each fixed representation, $\Tilde{x}$. Pointwise (or stratum-wise), their definition of the volume are sections of:
\begin{equation*}
    \Lambda^{max}H^1(F_g, \mathfrak{g}_{\tilde{x}})'\otimes \Lambda^{max}H^0 (F_g, \mathfrak{g}_{\tilde{x}}).
\end{equation*}

The volume on $H^1(F_g, \mathfrak{g}_{\tilde{x}})'$ can be seen from the short exact sequence:
$$\SES{B^1(F_g, \mathfrak{g}_{\tilde{x}})}{Z^1(F_g, \mathfrak{g}_{\tilde{x}})}{H^1(F_g, \mathfrak{g}_{\tilde{x}})}$$
to be:
$$(\eta(\mathfrak{g})^*)^g \otimes \eta(\mathfrak{l}^\perp_{\tilde{x}})$$
which we now tensor with the volume on $H^0 (F_g, \mathfrak{g}_{\tilde{x}})$:
$$(\eta(\mathfrak{g})^*)^g \otimes \eta(\mathfrak{l}^\perp_{\tilde{x}})\otimes \eta(\mathfrak{l}_{\tilde{x}})$$
and now splitting the volume $(\eta(\mathfrak{g})^*)^g \cong (\eta(\mathfrak{l})^*)^g \otimes (\eta(\mathfrak{l}^\perp)^*)^g$ gives:
$$(\eta(\mathfrak{l})^*)^g\otimes((\eta(\mathfrak{l}^\perp)^*)^g\otimes\eta(\mathfrak{l}^\perp_{\tilde{x}})\otimes \eta(\mathfrak{l}_{\tilde{x}}))$$
which reduces to:
\begin{equation}
w(\tilde{x}):=(\eta(\mathfrak{l})^*)^g\otimes((\eta(\mathfrak{l}^\perp)^*)^g\otimes\eta(\mathfrak{g}))
\end{equation}
Notice for $\Tilde{x}$ irreducible, $v$ is the Riemannian volume on $L_H^3$. Further, for $\Tilde{x}\in L_H^1$, $w$'s first tensorial component, $\eta(\mathfrak{l}_{\Tilde{x}})^g$ is the Riemannian volume on $T_{x}L_H^1$.
Let us recall how the Levi-Civita connection acts on top forms:
\begin{definition*}
    Fix a Riemannian n-manifold with Levi-Civita connection, $\nabla$, vector fields $Y, X_1,\dots X_n$ and a top form $\Omega$, we have the expression:
    $$\nabla_Y \Omega (X_1,\dots, X_n) = Y(\Omega (X_1,\dots, X_n))-\sum_{i=1}^n \Omega (X_1,\dots, X_{i-1}, \nabla_Y X_i, X_{i+1},\dots, X_n).$$
\end{definition*}
This covariant derivative expression shows the fact that the Riemannian volume associated to a Riemannian metric is covariant constant. We will denote this stratified Riemannian-volume-induced half-density on $L_H$ by $\xi\in |\Lambda|^{1/2}L_H$.

We now comment on the naturality of achieving covariant constant half-densities.  The polarization is maximally degenerate at this fibre -- all of the component functions achieve the maximum value of the trace and so have zero differential. Traditionally in geometric quantization, these functions would give rise to Hamiltonian vector fields which are in turn used to form a preferred -- and automatically covariant constant -- volume on the fibre. In \cite{JEFFREY1993509}, the a.t.d connection yielded their exact sequence (4.9, page 521) which bolstered the idea to use the Riemannian volume induced by $SU(2)$. On generic fibres, the volumes are covariant constant automatically. That is, the natural connection (on the 1-dimensional determinant line of said fibres) coincides with the Riemannian connection. In this sense, it is natural to ask for our volumes along $L_H$ to be covariant constant with respect to the Riemannian connection.  
\section{The stratified BKS pairing}
In this section we summarize the BKS pairing by recalling from theorem 5.1 of \cite{JEFFREY1993509} that for two Lagrangian subvarieties of a symplectic manifold, densities on each of these varieites can be paired to yield a density on their intersection given that their intersection is clean:
\begin{definition}
    \emph{(clean intersection)} The Lagrangian subvarieties, $L_1, L_2\subseteq (M, \omega)$ have clean intersection given that $T_xL_1 \cap T_x L_2 = T_x(L_1\cap L_2)$ for any $x\in L_1 \cap L_2$.
\end{definition}
This definition was the clear motive for its requirement in Jeffrey-Weitsman's \emph{very smooth} condition: the condition determining which $3-$manifolds their invariant is defined for. It is natural in the present context to ask for $L_{H_1}^1$ and $L_{H_2}^1$ to also have clean intersection -- though this requirement does not appear directly in their work and will now be treated more subtly. For the remainder of this section, $\Tilde{x}$ will be a non-trivial reducible flat connection. We have the decomposition:
$$H^1(N, \lieG_{\Tilde{x}})\cong H^1(N, \mathfrak{l}_{\Tilde{x}})\oplus H^1 (N, \mathfrak{l}_{\Tilde{x}}^\perp)$$
where the first summand has coefficients in the Lie algebra of the maximal torus -- that is, the untwisted or deRham cohomology. For the handlebody:
$$T_xL_{H_i}^1 \cong H^1(H_i, \mathfrak{l}_{\tilde{x}})\cong \mathfrak{l}_{\Tilde{x}}^g,$$
we have a setup which parallels that of the top stratum, with an exact sequence (see equation (5.1) on page 522 of \cite{JEFFREY1993509} where the explicit maps are given):
$$0\to H^1(H_1)\cap H^1(H_2) \to H^1(H_1)\oplus H^1(H_2)\to H^1(\Sigma^g) \to (H^1(H_1)\cap H^1(H_2))^*\to 0$$
compared with the sequence from Myer-Vietoris for $N=H_1\cup_\Sigma H_2$:
$$0\to H^1(N)\to H^1(H_1)\oplus H^1(H_2)\to H^1(\Sigma^g)\to (H^1(N))^*\to 0$$
suggests that we might instead require the following version of clean intersection:
\begin{definition}
    The moduli space of flat connections on both handlebodies of a given Heegard decomposition of $N$ which extend as flat connections to either handlebody, denoted $\mathscr{M}(N) := L_{H_1}\cap L_{H_2}$, has \emph{stratified clean intersection} if: 
    \begin{enumerate}
        \item for $x \in \mathscr{M}^3(N)$: $T_x \mathscr{M}(N) = H^1(N,\lieG_{\Tilde{x}}),$
        \item for $x\in \mathscr{M}^1(N)$: $T_x \mathscr{M}(N) = H^1(N,\mathfrak l_{\Tilde{x}}).$
    \end{enumerate}
    where $\mathscr{M}^i(N) = L_{H_1}^i\cap L_{H_2}^i$.
\end{definition}
As in \cite{JEFFREY1993509}, we have fixed a bi-invariant metric on the Lie algebra, and have a half-density $v_x\in \Lambda^{1/2}(\mathfrak{l}_{\Tilde{x}})$, which is constant on each stratum, and thus note:
$$T(N, \lieG_{\Tilde{x}}) \otimes v_x^{-2}= \zeta(\xi_1(x), \xi_2(x))$$
for $\Tilde{x}$ irreducible and:
$$T(N, \mathfrak{l}_{\Tilde{x}}) \otimes v_x^{-2}= \zeta(\xi_1(x), \xi_2(x))$$
for $\Tilde{x}$ reducible and non-trivial. Whilst these two torsions have different structure groups, we will denote their descent mutually as $\tau(N,x)$, independent of representative $\Tilde{x}$.   
\section{The 3-manifold invariants}
\begin{definition}
    Suppose $\mathscr{M}(N)$ has stratified clean intersection then:
    $$Z(N,k):=\sum_{i=0,1,3}\int_{\mathscr{M}^i(N)}<s_1,s_2>\zeta(\xi_1, \xi_2)$$
\end{definition}
The pairing of prequantum states remains as developed in \cite{JEFFREY1993509} and the pairing of the covariant constant half-densities yields the torsion, interpreted as discussed in the prior section. To summarize, we have:
\begin{theorem}
    Let $N$ be a 3-manifold such that $\mathscr{M}(N)$ has stratified clean intersection, then $Z(N,k)$ is an invariant coming from the BKS pairing of covariant constant states, and is given by:
    $$Z(N,k) = \sum_{i=0,1,3}\int_{\mathscr{M}^i(N)}e^{ikCS(x)}\tau(N,x)\otimes v_x^{-2}.$$
\end{theorem}
\noindent where $CS(x)$ is the Chern-Simons invariant of a connection determined by $x$, $\tau(N,x)$ is the torsion of the connection as discussed prior, and $v_x^{-2}\in |\Lambda|(Lie(Stab(x))^*)$.

In conclusion, for irreducible connections, this integral remains unchanged from the Jeffrey-Weitsman invariant. The implementation and study of the $L_H$'s stratification first yields a streamlining of the \enquote{very smooth} condition by removing their first and third condition. A more nuanced change appears with the amendment of requiring a \emph{stratified} clean intersection. A priori, while we still clearly get an invariant, the contribution from reducibles may change. For the examples of 3-manifolds discussed ($\mathbb{T}^3$, $\SSp^1\times\SSp^2$, $L(p,q)$) the contribution from lower strata are exactly the same (to this extent, our invariant appears to generalize that of Jeffrey-Weitsman). Of key note: for genus $g\geq 2$ we have $H^1(\SSp^1\times \Sigma^g, \lieG_{\Tilde{x}})$ is of dimesnion $6g-1$ whereas $H^1(\SSp^1\times \Sigma^g, \lieG_{\Tilde{x}})$ is of dimension $2g+1$ (one can check this via the Kunneth formula yielding the first deRham cohomology, or by computing the Zariski tangent space to the representation variety as we did earlier), when the representation is non-trivial and reducible. In this case we see the \emph{very smooth} condition would preclude contributions from any such reducibles. Using the reduced torsion as we have required here aligns perfectly with the known minimial Heegard decomposition of $\SSp^1 \times \Sigma^g$ of genus precisely $2g+1$ (see \cite{heegardSplittings} for a classification of such decompositions). 
\nocite{*}
\printbibliography

\end{document}